\documentclass[10pt]{article}
\usepackage[utf8]{inputenc}
\usepackage{setspace}
\title{A note on forcing triples with no forcing pairs}
\author{Nikola Spasić}
\date{}

\usepackage{subfiles}
\usepackage{enumitem}
\usepackage{amsthm}
\usepackage{xcolor}
\usepackage{indentfirst}
\usepackage{thmtools,thm-restate}
\usepackage{url}

\usepackage[square,sort,comma,numbers]{natbib}

\newtheorem{theorem}{Theorem}[section]
\newtheorem{definition}{Definition}[section]

\newtheorem{proposition}{Proposition}[section]

\newtheorem{lemma}[theorem]{Lemma}

\usepackage{tikz}
\usepackage[ruled,vlined,noend]{algorithm2e}
\usepackage{natbib}
\usepackage{amsmath}
\usepackage{amssymb}
\usepackage{bbm}
\usepackage{graphicx}
\usepackage{caption}
\usepackage{geometry}
\usepackage{environ}
\NewEnviron{restatethis}[1]{%
  \expandafter\xdef\csname restatethis@#1\endcsname{%
    \unexpanded\expandafter{\BODY}%
  }%
  \renewcommand{\thetheorem}{\ref{#1}}%
  \begin{theorem}\BODY\end{theorem}%
}
\newcommand{\restate}[1]{%
  \begin{theorem}\label{#1}\csname restatethis@#1\endcsname\end{theorem}%
}

\usepackage{tikz}
\usepackage{scalerel}
\usepackage{pict2e}
\usepackage{tkz-euclide}
\usetikzlibrary{calc}
\usetikzlibrary{patterns,arrows.meta}
\usetikzlibrary{shadows}
\usetikzlibrary{external}

\usepackage{pgfplots}
\pgfplotsset{compat=newest}
\usepgfplotslibrary{statistics}
\usepgfplotslibrary{fillbetween}

\usepackage{xcolor}

\begin{document}

\maketitle

\begin{abstract}
\noindent
Chung, Graham and Wilson defined a set of graphs $\mathcal{H}$ to be \emph{forcing}, if any sequence of graphs $\{G_n\}_{n \geq 0}$ with $|G_n| = n$ must be quasirandom, whenever $hom(H, G_n)= (p^{|E(H)|}+o(1))n^{|V(H)|}$ for every $H \in \mathcal{H}$ and some constant $p \in (0, 1)$.
Answering a question of Horn, attributed to Graham, a forcing set of three graphs is constructed such that no two of the three graphs are forcing as a pair.

\end{abstract}

\section{Introduction}

Since the introduction of the notion of quasirandomness by Chung, Graham and Wilson \cite{CGW}, building upon the work of Thomason \cite{T-87}, a great deal of work was done to find properties equivalent to quasirandomness \cite{CFS-10}\cite{CHPS-11}\cite{HPS-11}\cite{SS-97}\cite{SS-2007}\cite{ST-04}. 
To be more precise, a part of what Chung, Graham and Wilson showed is contained in the following theorem.

\begin{theorem}
    For a sequence of graphs $\{G_n\}_{n \geq 0}$ such that $|G_n| = n$ and a fixed constant $p \in (0, 1)$, the following are equivalent:
    \begin{itemize}[leftmargin=0.5in]
        \item [(P0) -] $\mathcal{G}$ is quasirandom.
        \item [(P1) -] $\forall U \subset V(G_n), |E(G_n[U])| = p\binom{|U|}{2} + o(n^2)$.
        \item [(P2) -] $|E(G_n)| = (p+o(1))n^2$, $\lambda_1 = (p+o(1))n$ and $\lambda_i = o(n)$ for $i \neq 1$, where $\lambda_1 \geq \lambda_2 \geq \dots \geq \lambda_n$ is the spectrum of $G_n$.
        \item [(P3) -] For every graph $H,$ $ hom(H, G_n) = (p^{|E(H)|}+o(1))n^{|V(H)|}.$
        \item [(P4) -] $hom(e, G_n) = (p+o(1))n^2$ and $hom(C_4, G_n) = (p^4+o(1))n^4$.
    \end{itemize}
\end{theorem}

\noindent
Let us focus on the last two conditions from the theorem above. Informally, $\mathit{P3}$ requires correct homomorphism counts for all graphs, while  $\mathit{P4}$, seemingly, asks for much less. 
Indeed, it asks for homomorphism counts to be correct just for graphs from the set $\mathcal{H} := \{e, C_4\}$,  but it still ensures quasirandomness of the sequence. 
In particular, this implies that if the homomorphism counts are correct for graphs in $\mathcal{H}$, then all graphs are forced to have correct homomorphism counts. 
This motivates the following definition of Chung, Graham and Wilson.

\begin{definition}
    A set of graphs $\mathcal{H}$ is forcing if for any sequence of graphs $\{G_n\}_{n \geq 0}$, such that $|G_n| = n$, and any fixed constant $p \in (0, 1)$, the following are equivalent:
    \begin{itemize}
    [leftmargin=0.5in]
        \item[-] $\forall H \in \mathcal{H}, hom(H, G_n) = (p^{|E(H)|}+o(1))n^{|V(H)|}$.
        \item [(P0) -] $\mathcal{G}$ is quasirandom.
    \end{itemize}
\end{definition}

\noindent
By showing that $\mathit{P4}$ implies $\mathit{P0}$, Chung, Graham and Wilson showed that the set $\{e, C_4\}$ is forcing. 
Furthermore, they also showed that the sets $\{e, C_{2a}\}$ and $\{P_2, K_{2, a}\}$ are forcing for any $a \geq 2$, already showing many forcing pairs of graphs.\footnote{Slightly abusing the notation, we shall use the statements ``the set $\mathcal{H}=\{H_i : i \in I\}$ is forcing" and ``the graphs $H_i$ for $i \in I$ are forcing" interchangeably.} 
Skokan and Thoma \cite{ST-04} showed that $\{e, K_{a, b}\}$ is forcing for any $a, b \geq 2$, and asked whether $\{e, B\}$ is forcing for any bipartite graph $B$ that contains a cycle. 
Conlon, Fox and Sudakov \cite{CFS-10} conjectured that this is true, naming it the \textit{forcing conjecture}. 
Moreover, they showed that $\{e, H\}$ is forcing for any bipartite graph $H$ such that each vertex set has at least two vertices, and one vertex set has two vertices that are completely connected to the other part. 
As for forcing pairs of non-bipartite graphs, Conlon, H\`an, Person and Schacht \cite{CHPS-11} showed that for every $a$, there is a graph $H$ such that $\{K_a, H\}$ is forcing. 
Continuing in this direction, H\`an, Person and Schacht \cite{HPS-11} showed that for every graph $F$, there is a graph $H$ such that $\{F, H\}$ is forcing.\\

The above-mentioned results deal with forcing \emph{pairs} rather than with sets of more than two graphs. An obvious question, if one wishes to consider larger sets of graphs, is whether there is a forcing triple of graphs that contains no forcing pair. This question appears in an online collection of problems collated by Paul Horn \cite{PH-Q}, where it is attributed to Ron Graham.


This question is answered in the affirmative in this paper. The main result is the following theorem.

\begin{theorem}
\label{theorem:main}
    There are forcing triples with no forcing pairs.
\end{theorem}

\noindent
In fact, three infinite families of such triples will be constructed, in Theorem \ref{theorem:explicit} below, each parametrised by an arbitrary connected non-bipartite graph $T$.
The first construction, $\mathcal{H}_1(T)$, is the simplest one; the second construction, $\mathcal{H}_2(T)$, will have $e$ as one of the graphs; the third construction, $\mathcal{H}_3(T)$, will have $T$ as one of the graphs.
Actually, the graphs in $\mathcal{H}_3(T)$ have two additional properties worth mentioning. 
Firstly, they  all have the same chromatic number as $T$. Secondly, their number of vertices is linear in $|T|$, which contrasts with the exponential dependency in known constructions for a graph $H$ such that the pair $\{T, H\}$ is forcing \cite{CHPS-11}\cite{HPS-11}. \\

The question above is understood to ask for connected graphs, and all graphs used in forming the triples will be connected. If disconnected graphs are allowed, then the question becomes much easier. For example, from what will be shown, it will be clear that if $\{B_1, B_2\}$ is a forcing set consisting of two bipartite graphs, and if $a$ is a large enough integer, then $\{B_1, B_2K_a, K_a\}$ is a forcing triple with no forcing pairs. Here $F_1F_2$ means a disjoint union of graphs $F_1$ and $F_2$. Indeed, since $hom(F_1F_2, G) = hom(F_1, G)hom(F_2,G)$, the homomorphism counts of $K_a$ and $B_2K_a$ determine the homomorphism count of $B_2$. This shows that the triple mentioned is forcing, while the fact that pairs used are not forcing follows from Proposition \ref{prop:condition}.\\

The structure of the paper is as follows. Section \ref{section:notation} introduces the necessary notation, while Section \ref{section:construction} presents the construction of all the triples. In Section \ref{section:forcing} we show that these triples are forcing, and in Section \ref{section:non-forcing} we show that none of them contains a forcing pair. 
Finally, in two appendices, we present the version we shall need of a well-known result about near equality in the Cauchy-Schwarz inequality, and verify some simple inequalities that are needed in Section \ref{section:non-forcing}. 
The main result, Theorem \ref{theorem:main}, is a direct consequence of Theorem \ref{theorem:explicit}.

\section{Notation}
\label{section:notation}
For any graph $F$, let $n(F)$, $m(F)$ and $b(F)$ be the number of vertices, the number of edges,
and the number of edges in the largest bipartite subgraph of $F$, i.e. max-cut$(F)$, respectively. 
Fix $T$ to be an arbitrary  connected non-bipartite graph, and define $N := n(T)$, $M := m(T)$ and $B := b(T)$. \\

Let $F$ and $G$ be arbitrary graphs, and denote by $t_G(F) = t(F, G)$ the probability that a random map ${\phi : V(F) \to V(G)}$ is a graph homomorphism.
Throughout, when we talk about the copies of a graph inside another graph, we are talking about graph homomorphisms. 
If $\mathcal{G} = \{G_n\}_{n\geq 0}$ is a sequence of graphs, let $t(F, n) := t_{G_n}(F)$, and let $t(F) = t_{\mathcal{G}}(F) := \lim_{n \to \infty} t(F, n)$, if the limit exists.
Also, if $t(e)$ exists for some sequence, it will be called the density of the sequence. In particular, every quasirandom sequence of graphs has a density.
All the sequences $\mathcal{G}$ are assumed to be increasing in order. In fact, we shall assume that $|V(G_n)| = n$, which can be done without loss of generality.\\

To put the definition of forcing into this notation, a set of graphs $\mathcal{H}$ is said to be forcing if for any sequence of graphs $\mathcal{G}$, and any fixed constant $p \in (0, 1)$, the fact that $t(H) = p^{|E(H)|}$ for all graphs $H \in \mathcal{H}$ implies that the sequence $\mathcal{G}$ is quasirandom. 
Having fixed a sequence $\mathcal{G}$, for a graph $F$,  let $\#F := hom(F, G_n)$ be the number of copies  of $F$ in $G_n$ (which makes $\#F$ a function of $n$).
In particular, if $t(F)$ exists, then $\#F = (t(F)+o(1))n^{|V(F)|}$. 
\\

Since the graphs that will be counted are implicitly labelled (some of the labels are shown in the Figure \ref{fig:graphs}), one can look at joint copies of two graphs, say $F_1$ and $F_2$.
By the term joint copy, we mean a pair of homomorphisms, one from $F_1$ into $G_n$ and the other from $F_2$ into $G_n$, where the homomorphisms agree on vertices with the same labels. 
The number of such joint copies of $F_1$ and $F_2$ will be denoted by $\#\{F_1, F_2\}$. Furthermore, the number of joint copies of $F_1$ and $F_2$ once a copy of $F_2$ has been fixed will be denoted by $\#\{F_1 \mid F_2\}$.
Finally, for a graph $F$, and a fixed sequence $\mathcal{G}$, or more precisely a fixed $G_n \in \mathcal{G}$, $\sum_{F}$ stands for the sum over the copies of $F$ in $G_n$.\\

Before proceeding further, in order to reinforce the notation consider the following. 
For any graphs $F$ and $H$, one has $$\#\{F, H\}=\sum_H \# \{F \mid H\}.$$
Indeed, this is just the law of total probability. Moreover, if $H \subset F$, then $\#F = \#\{F, H\}$.

\section{Construction}
\label{section:construction}
In order to describe the construction of the graphs we shall use, we introduce the following notation. Firstly, for a graph $F$, and a set $I$, define the \emph{doubling} of $F$ over $I$, denoted by $db_I(F)$, to be the graph obtained by taking two disjoint copies of $F$ and identifying vertices with the same label if that label is from the set $I$. 
In $db_I(F)$ vertices from $F$ with labels in $I$ preserve their labels, while other vertices are labelled arbitrarily. 
The brackets in this notation will often be dropped. For example, if $I = \{a, b\}$, we shall write $db_{a,b}$ instead of $db_{\{a, b\}}$. 
Secondly, for a graph $F$ that has a vertex labelled $0$ and a positive integer $k$, define $F^{(k)}$ to be a graph consisting of a copy of $F$ and $k$ disjoint new vertices, labelled $1$ through $k$, with new edges connecting these vertices to the vertex with label $0$ in the copy of $F$, which preserves its label. 
This notation will be abused, in that $F'$ and $F''$ stand for $F^{(1)}$ and $F^{(2)}$, respectively. \\

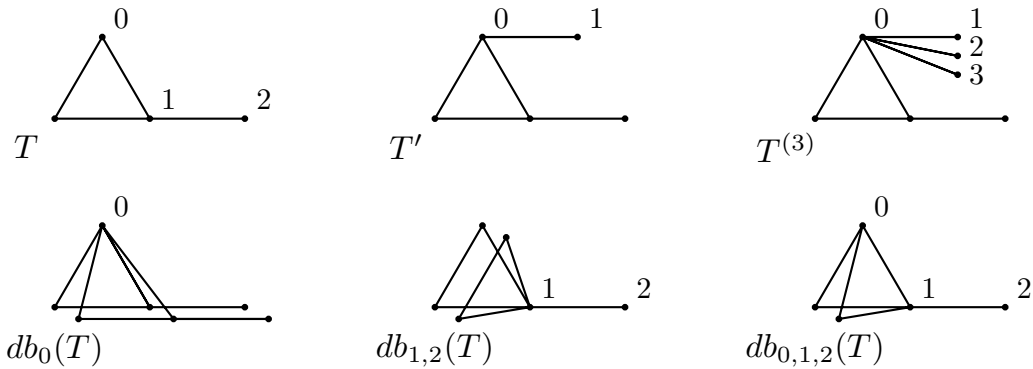
\begin{figure}[h!]
    \centering
    \def\ra{0.5}
\def\rb{0.57735026919}
\def\rc{1/2}
\def\rd{0.15}
\def\re{0.012}
\def\rf{0.45}
\def\rg{0.10}
\def\rt{0.125}

\begin{tikzpicture}[thick,scale=2.5, every node/.style={transform shape}]

\draw (1, -1) -- ++(-1, 0) -- ++(0.25, 0.43301270189) -- ++(0.25, -0.43301270189);

\draw (3, -1) -- ++(-1, 0) -- ++(0.25, 0.43301270189) -- ++(0.5, 0) -- ++(-0.5, 0)  -- ++(0.25, -0.43301270189);

\draw (5, -1) -- ++(-1, 0) -- ++(0.25, 0.43301270189) -- ++(0.5, 0) -- ++(-0.5, 0) -- ++(0.5, -\rg) -- ++(-0.5, \rg) -- ++(0.5, -2 * \rg) -- ++(-0.5, 2 * \rg)  -- ++(0.25, -0.43301270189);

\draw (1, -2) -- ++(-1, 0) -- ++(0.25, 0.43301270189) -- ++(0.25, -0.43301270189) -- ++(-0.25, +0.43301270189) -- ++(0.25 + \rt, -0.43301270189 - \rt/2) -- ++(0.5, 0) -- ++(-1, 0) -- ++(0.25 - \rt, +0.43301270189 + \rt/2);

\draw (3, -2) -- ++(-1, 0) -- ++(0.25, 0.43301270189) -- ++(0.25, -0.43301270189) -- ++(-0.25 + \rt, +0.43301270189 - \rt/2) -- ++(-0.25, -0.43301270189) -- ++(0.5 - \rt, 0 + \rt/2) ;

\draw (5, -2) -- ++(-1, 0) -- ++(0.25, 0.43301270189) -- ++(0.25, -0.43301270189) -- ++(-0.5 + \rt, +0 - \rt/2) -- ++(0.25 - \rt, +0.43301270189 + \rt/2);


\filldraw[black] (0, -1) circle (\re) node[scale=\rf] at +(\rg,\rg) {};

\filldraw[black] (1 * 0.5 , - 2 * 0.5) circle (\re) node[scale=\rf] at +(\rg, \rg) {1};

\filldraw[black] (2 * 0.5 , - 2 * 0.5) circle (\re) node[scale=\rf] at +(\rg,\rg) {2};

\filldraw[black] (4 * 0.5 , - 2 * 0.5) circle (\re) node[scale=\rf] at +(\rg,\rg) {};

\filldraw[black] (5 * 0.5 , - 2 * 0.5) circle (\re) node[scale=\rf] at +(\rg,\rg) {};

\filldraw[black] (6 * 0.5 , - 2 * 0.5) circle (\re) node[scale=\rf] at +(\rg,\rg) {};

\filldraw[black] (8 * 0.5 , - 2 * 0.5) circle (\re) node[scale=\rf] at +(\rg,\rg) {};

\filldraw[black] (9 * 0.5 , - 2 * 0.5) circle (\re) node[scale=\rf] at +(\rg,\rg) {};

\filldraw[black] (10 * 0.5 , - 2 * 0.5) circle (\re) node[scale=\rf] at +(\rg,\rg) {};

\filldraw[black] (0.25 + 0 * 0.5 , - 1 + 0.43301270189) circle (\re) node[scale=\rf] at +(\rg,\rg) {0};

\filldraw[black] (0.25 + 4 * 0.5 , - 1 + 0.43301270189) circle (\re) node[scale=\rf] at +(\rg,\rg) {0};

\filldraw[black] (0.25 + 8 * 0.5 , - 1 + 0.43301270189) circle (\re) node[scale=\rf] at +(\rg,\rg) {0};

\filldraw[black] (0.25 + 5 * 0.5 , - 1 + 0.43301270189) circle (\re) node[scale=\rf] at +(\rg,\rg) {1};

\filldraw[black] (0.25 + 9 * 0.5 , - 1 + 0.43301270189 - 0 * \rg) circle (\re) node[scale=\rf] at +(\rg,1*\rg) {1};

\filldraw[black] (0.25 + 9 * 0.5 , - 1 + 0.43301270189 - 1 * \rg) circle (\re) node[scale=\rf] at +(\rg,0.5*\rg) {2};

\filldraw[black] (0.25 + 9 * 0.5 , - 1 + 0.43301270189 - 2 * \rg) circle (\re) node[scale=\rf] at +(\rg,0*\rg) {3};

\filldraw[black] (0 + 0, -2) circle (\re) node[scale=\rf] at +(\rg,\rg) {};

\filldraw[black] (1 * 0.5 + 0, - 2) circle (\re) node[scale=\rf] at +(\rg, \rg) {};

\filldraw[black] (2 * 0.5 + 0, - 2) circle (\re) node[scale=\rf] at +(\rg,\rg) {};

\filldraw[black] (0.25 + 0, - 2 + 0.43301270189) circle (\re) node[scale=\rf] at +(\rg,\rg) {0};

\filldraw[black] (1 * 0.5 + 0 + \rt, - 2 - \rt/2) circle (\re) node[scale=\rf] at +(\rg, \rg) {};

\filldraw[black] (2 * 0.5 + 0 + \rt, - 2 -\rt/2) circle (\re) node[scale=\rf] at +(\rg,\rg) {};

\filldraw[black] (0 + 0 + \rt, -2 - \rt/2) circle (\re) node[scale=\rf] at +(\rg,\rg) {};

\filldraw[black] (0 + 2, -2) circle (\re) node[scale=\rf] at +(\rg,\rg) {};

\filldraw[black] (1 * 0.5 + 2, - 2) circle (\re) node[scale=\rf] at +(\rg, \rg) {1};

\filldraw[black] (2 * 0.5 + 2, - 2) circle (\re) node[scale=\rf] at +(\rg,\rg) {2};

\filldraw[black] (0.25 + 2, - 2 + 0.43301270189) circle (\re) node[scale=\rf] at +(\rg,\rg) {};

\filldraw[black] (0 + 2 + \rt, -2 - \rt/2) circle (\re) node[scale=\rf] at +(\rg,\rg) {};

\filldraw[black] (0.25 + 2 + \rt, - 2 + 0.43301270189 - \rt/2) circle (\re) node[scale=\rf] at +(\rg,\rg) {};

\filldraw[black] (0 + 4, -2) circle (\re) node[scale=\rf] at +(\rg,\rg) {};

\filldraw[black] (1 * 0.5 + 4, - 2) circle (\re) node[scale=\rf] at +(\rg, \rg) {1};

\filldraw[black] (2 * 0.5 + 4, - 2) circle (\re) node[scale=\rf] at +(\rg,\rg) {2};

\filldraw[black] (0.25 + 4, - 2 + 0.43301270189) circle (\re) node[scale=\rf] at +(\rg,\rg) {0};

\filldraw[black] (0 + 4 + \rt, -2 - \rt/2) circle (\re) node[scale=\rf] at +(\rg,\rg) {};


\node[scale=\rc] at ($(0, -1) + (-\rd, -\rd)$) {$T$};

\node[scale=\rc] at ($(2, -1) + (-\rd, -\rd)$) {$T'$};

\node[scale=\rc] at ($(4, -1) + (-\rd, -\rd)$) {$T^{(3)}$};

\node[scale=\rc] at ($(0, -2) + (0, -\rd * 1.5)$) {$db_0(T)$};

\node[scale=\rc] at ($(2, -2) + (0, -\rd * 1.5)$) {$db_{1,2}(T)$};

\node[scale=\rc] at ($(4, -2) + (0, -\rd * 1.5)$) {$db_{0,1,2}(T)$};


\end{tikzpicture}
    \caption{An example of the constructions described.}
    \label{fig:operations}
\end{figure}


Returning to the example from the previous section, if $H \subset F$, for some graphs $H$ and $F$, then $$ \#F = \#\{F, H\}=\sum_H \# \{F \mid H\}.$$
Furthermore, in that case, one has 
$$ \#\{db_{V(H)}(F) \mid H\} = \#\{F \mid H\}^2,$$
which implies
$$ \#db_{V(H)}(F) = \sum_H \#\{F \mid H\}^2.$$\\

Before constructing the graphs, we need to define the labeling of graphs that will be used. 
Let $v$ be a graph consisting of a single vertex, labelled $0$; let $\overline{e}$ be a graph consisting of two disconnected vertices, labelled $1$ and $2$; and let $e$ be a graph consisting of a single edge with vertices labelled $0$ and $1$, i.e. $e = v'$.
Finally, assume that $T$, the arbitrary connected non-bipartite graph, has a vertex labelled $0$. 
Having defined the above, the following result implies Theorem \ref{theorem:main}.

\begin{theorem}
\label{theorem:explicit}
    Let $\mathcal{H}$ be any of the following 
    \begin{itemize}
        \item [-] $\mathcal{H}_1(T) := \{T', db_{0,1}(T'), C_{2N}\}$,
        \item [-] $\mathcal{H}_2(T) := \{e, T', db_{1,2}(db_0^2(T)'')\}$
        \item [-] or $\mathcal{H}_3(T) := \{T, T', db_{1, 2}(db_0(T^{(N)})'')\}$.
    \end{itemize}    
    Then $\mathcal{H}$ is forcing, and no two graphs in $\mathcal{H}$ are forcing.
\end{theorem}

The above theorem is a direct consequence of Lemmas \ref{lemma:triple_first}, \ref{lemma:triple_second} and \ref{lemma:triple_third}, from Section \ref{section:forcing}, which show the triples are forcing, and Theorem \ref{theorem:non-forcing} from Section \ref{section:non-forcing} which shows that the pairs within the triples are not forcing.

\begin{figure}[h!]
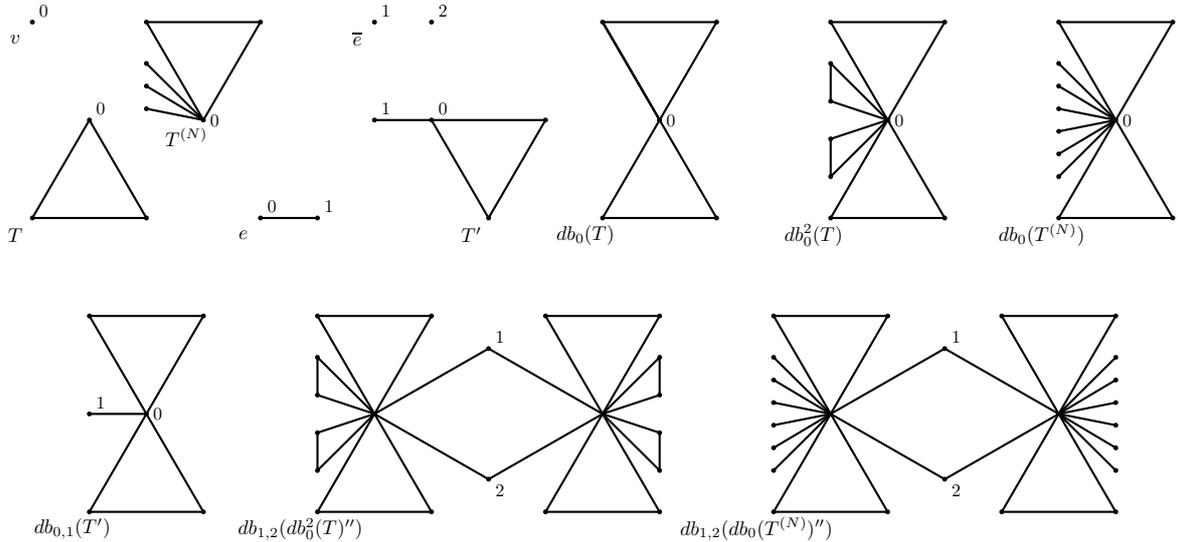

  \centering
  \subfile{figures/graphs.tex}
  \caption{The graphs used, in the case where $T$ is a triangle $K_3$.}
  \label{fig:graphs}
\end{figure}


\section{Forcing triples}
\label{section:forcing}

As mentioned in the introduction, the best known forcing set is the pair $\{e, C_4\}$, introduced in \cite{CGW}, and in the same paper it was shown that each pair $\{e, C_{2k}\}$ for $k > 2$ is forcing as well.
All the triples presented in this note force quasi-randomness by forcing the homomorphism counts of graphs belonging to such a pair. 
Before we prove the lemmas, consider the following proposition, which will reinforce the notation.
\begin{proposition}
    For a sequence of graphs $\mathcal{G}$, and a constant $p \in (0, 1)$, if $\#e \geq pn^2 + o(n^2)$ then $\#C_4 \geq p^4n^4 + o(n^4)$.
\end{proposition}
\begin{proof}
    For convenience, label $C_4$ as $C_4 = db_{1,2}(v'')$. Note that $P_2 = v''$ is a path of length 2. 
    Now, observe that if one fixes $v$, then there are $\#\{e \mid v\}$ ways to extend it to an edge, and therefore there are $\#\{e \mid v\}^2$ ways to extend $v$ into $v''$. 
    In other words,
    $$ \#\{e \mid v\}^2 = \#\{P_2 \mid v\}.$$
    Thus, one can use Jensen's inequality to show 
    $$
    \left(\frac{\#e}{\#v}\right)^2 
    = \left(\frac{1}{\#v}\sum_v{\#\{e \mid v\}}\right)^2 
    \leq \frac{1}{\#v}\sum_v{\#\{e \mid v\}}^2
    = \frac{1}{\#v}\sum_v{\#\{P_2 \mid v\}}
    = \frac{\#P_2}{\#v}.
    $$
    Since $\#e \geq pn^2 + o(n^2)$ and $\#v = n$, the above inequality implies $\#P_2 \geq p^2n^3 + o(n^3)$. It is now clear that a similar bound can be obtained using $P_2$ and $C_4$. Recall that $\overline{e}$ is a graph consisting of two disconnected vertices, labelled $1$ and $2$.
    By definition, 
    $$ \#\{P_2 \mid \overline{e}\}^2 = \#\{C_4 \mid \overline{e}\},$$
    so again Jensen's inequality provides the bound    
    $$
    \left(\frac{\#P_2}{\#\overline{e}}\right)^2 
    = \left(\frac{1}{\#\overline{e}}\sum_{\overline{e}}{\#\{P_2 \mid \overline{e}\}}\right)^2 
    \leq \frac{1}{\#\overline{e}}\sum_{\overline{e}}{\#\{P_2 \mid \overline{e}\}}^2
    = \frac{1}{\#\overline{e}}\sum_{\overline{e}}{\#\{C_4 \mid \overline{e}\}}
    = \frac{\#C_4}{\#\overline{e}}.
    $$
    Thus, again, $\#P_2 \geq p^2n^3 + o(n^3)$ and $\#\overline{e} = n^2$ imply $\#C_4 \geq p^4n^4 + o(n^4)$, and the proof is complete.
\end{proof}

The above argument also shows that if $\#C_4 \leq p^4n^4 + o(n^4)$, then $\#e \leq pn^2 + o(n^2)$. In fact, these kinds of inequalities must hold between any two graphs $H_1, H_2$ such that $\{H_1, H_2\}$ is forcing. Indeed, a special case of this statement follows from Proposition \ref{prop:path} from Section \ref{section:non-forcing}. For now, it suffices to note that $\{e, C_{2a}\}$ for $a \geq 2$ is forcing \cite{CGW}, and that, as in the above example, having $\#C_{2a} \leq (p^{2a}+o(1))n^{2a}$ implies $\#e \leq (p+o(1))n^2$. \\

Keeping this in mind, we can start proving the forcing part of Theorem \ref{theorem:explicit}. We break the proof into three lemmas, each corresponding to one of the triples mentioned in the statement of Theorem \ref{theorem:explicit}, i.e. $\mathcal{H}_1$, $\mathcal{H}_2$ and $\mathcal{H}_3$. 

\begin{lemma}
\label{lemma:triple_first}
The set $\mathcal{H}_1(T)=\{T', db_{0,1}(T'), C_{2N}\}$ is forcing.
\end{lemma}

\begin{proof}
Fix some $p \in (0, 1)$, and some sequence of graphs $\mathcal{G}$. Assume that $t(T')$, $t(db_{0,1}(T'))$ and $t(C_{2N})$ are equal to $p^{M+1}$, $p^{2M+1}$ and $p^{2N}$, respectively. By Jensen's inequality

$$\left(\frac{1}{\#e} \#T'\right)^2 
= \left(\frac{1}{\#e} \sum_{e} \#\{T' \mid e\}\right)^2 
\leq \frac{1}{\#e} \sum_{e} (\#\{T' \mid e\})^2 
= \frac{1}{\#e} \sum_{e} \#\{db_{0,1}(T') \mid e\} 
= \frac{1}{\#e} \#db_{0,1}(T'). $$
This implies that $\#e \geq \frac{(\#T')^2}{\#db_{0,1}(T')}$. But since $\#T' = (p^{M+1}+o(1))n^{N+1}$ and $\#db_{0,1}(T') = (p^{2M+1}+o(1))n^{2N}$, this means that $\#e \geq (p+o(1))n^2$. 
On the other hand,  $ \# C_{2N} = (p^{2N}+o(1))n^{2N}$ provides the corresponding upper bound $\#e \leq (p+o(1))n^2$. Therefore $t(e) = p$, and thus $\mathcal{G}$ is quasirandom, since $\{e, C_{2N}\}$ is forcing.
\end{proof}

Before proceeding to Lemmas \ref{lemma:triple_second} and \ref{lemma:triple_third}, we state an auxiliary lemma, Lemma \ref{lemma:jensen} below, since it will be needed for the proofs of both lemmas. 
This lemma is just a minor variation of the inequalities already proved by Kohayakawa, R\"{o}dl and Sissokho \cite{KRS_04} and Skokan and Thoma \cite{ST-04}. For this reason, we do not prove it here, but for the convenience of the reader we provide a proof in Appendix \ref{section:jensen}.
Informally, the lemma states that in order for Jensen's inequality to be roughly sharp the set of points that take values that are not close to the average must have negligible weight.

\begin{restatable}{lemma}{jensen}
\label{lemma:jensen}
For every $\delta > 0$ and every integer $k \geq 2$ there exists an $\epsilon > 0$ such that for any non-negative $a_1, \dots, a_n$ and $p_1, \dots, p_n$ such that
$ \sum_{i=0}^{n} p_i = 1, $ and for any $A > 0$, the following holds. If
$$ \sum_{i=0}^{n} p_i a_i \geq A(1 - \epsilon), $$ 
and 
$$ \sum_{i=0}^{n} p_i a_i^k \leq A^k(1 + \epsilon) ,$$
then the set $I := \{0 \leq i \leq n : |a_i - A| \geq \delta A \}$ satisfies the condition 
$$ \sum_{i \in I} p_i \leq \delta. $$
\end{restatable}

\noindent
Armed with this lemma, we now prove the remainder of the forcing part of Theorem \ref{theorem:explicit}.

\begin{lemma}
\label{lemma:triple_second}
The set $\mathcal{H}_2(T) = \{e, T', db_{1,2}(db_0^2(T)'')\}$ is forcing.
\end{lemma}

A sketch of the proof is as follows. The rough sharpness of the Jensen's inequality between $\#T'$ and $\#db_0(T)'$ will be used to apply Lemma \ref{lemma:jensen}. 
This will show that almost all copies of $e$ are part of a correct number of copies of $T'$, i.e. for almost all copies of $e$ the value $\#\{T' \mid e\}$ is correct. 
After that, one can provide an upper bound on $\#C_4$, since almost all copies of $C_4$ can be extended to a correct number of copies of $db_{1,2}(db_0^2(T)'')$.

\begin{proof}
Fix some $p \in (0, 1)$, and some sequence of graphs $\mathcal{G}$. Assume that $t(e)$, $t(T')$ and $t(db_{1,2}(db_0^2(T)''))$ are equal to $p$, $p^{M+1}$ and $p^{8M+4}$, respectively. 
Similarly to what was done in the previous lemma, Jensen's inequality implies
$$ \left(\frac{1}{\#v}\#db_0(T)'\right)^2 \leq \frac{1}{\#v} \#db_0^2(T)''$$
and
$$ \left(\frac{1}{\#{\overline{e}}}\#db_0^2(T)''\right)^2 \leq \frac{1}{\#{\overline{e}}} \# db_{1,2}(db_0^2(T)'').$$
Therefore, since $\# db_{1,2}(db_0^2(T)'') = (p^{8M+4}+o(1))n^{8N-4}$, it follows that $\# db_0^2(T)'' \leq (p^{4M+2}+o(1))n^{4N-1}$, and finally $\#db_0(T)' \leq (p^{2M+1}+o(1))n^{2N}$. 
Now, note that 
$$ \sum_{v}\frac{\#\{e\mid v\}}{\#e} = \frac{\#e}{\#e} = 1,$$
$$ \sum_{v} \frac{\#\{e \mid v\}}{\#e} \#\{T \mid v\} = \frac{\#T'}{\#e} = (p^M+o(1))n^{N-1},$$
and 
$$ \sum_{v} \frac{\#\{e \mid v\}}{\#e} (\#\{T \mid v\})^2 = \frac{\#db_0(T)'}{\#e} \leq (p^{2M}+o(1))n^{2N-2}.$$
In particular, the conditions of Lemma \ref{lemma:jensen} are met, 
where the weights $p_i$ from the lemma are the proportions $\frac{\#\{e\mid v\}}{\#e}$, and the values $a_i$ from the lemma are the conditional counts $\#\{T \mid v\}$.
Applying Lemma \ref{lemma:jensen} gives us that the set $B_v$ of the ``bad" vertices, i.e. those vertices that are not in $(p^{M}+o(1))n^{N-1}$ copies of $T$, has weight $o(1)$. 
In other words, $\sum_{v \in B_v} \frac{\#\{e \mid v\}}{\#e} = o(1)$, or equivalently $\sum_{v \in B_v} \#\{e \mid v\} = o(n^2)$. 
Observe also, that for any $v \not\in B_v$, not only does $\#\{T \mid v\} = (p^{M}+o(1))n^{N-1}$ but also $\#\{db_0^2(T) \mid v\} = (p^{4M}+o(1))n^{4N-4}$. \\

Denote by $B_e$ the set of ``bad" edges, i.e. those edges that have at least one vertex in $B_v$. 
Note that $|B_e| = o(n^2)$, and thus, since each edge can be in at most $O(n^2)$ copies of $C_4$, there are $o(n^4)$ copies of $C_4$ that contain an edge in $B_e$. 
This implies that there are $o(n^4)$ copies of $C_4$ that contain a vertex in $B_v$. 
On the other hand, each of the remaining $\#C_4 - o(n^4)$ copies of $C_4$ has no vertices in $B_v$, and thus forms $((p^{4M}+o(1))n^{4N-4})^2$ unique copies of $db_{1,2}(db_0^2(T)'')$. 
Putting this together yields

$$ (\# C_4 - o(n^4))((p^{4M}+o(1))n^{4N-4})^2 \leq \#db_{1,2}(db_0^2(T)'') = (p^{8M+4}+o(1))n^{8N-4},$$
which implies that $\#C_4 \leq (p^4+o(1))n^4$.  This, together with the fact that $\#e = (p+o(1))n^2$, implies $\mathcal{G}$ is quasirandom, since $\{e, C_4\}$ is forcing.
\end{proof}

\begin{lemma}
\label{lemma:triple_third}
The set $\mathcal{H}_3(T) = \{T, T', db_{1,2}(db_0(T^{(N)})'')\}$ is forcing.
\end{lemma}

In the proof of this lemma, approximate equality in Jensen's inequality will be used to apply Lemma \ref{lemma:jensen} twice.
This will result in our proving that for almost all vertices $v$, both $\#\{e \mid v\}$ and $\#\{T \mid v\}$ are of roughly the expected size. 
After that, a lower bound on $\#e$ and an upper bound on $\#C_4$ is provided. The former follows from the fact that counts $\#\{e \mid v\}$ are approximately correct for almost all $v$, while the latter comes from the fact that almost all copies of $C_4$ can be extended to the correct number of copies of $db_{1,2}(db_0(T^{(N)})'')$.

\begin{proof}
Fix some $p \in (0, 1)$, and some sequence of graphs $\mathcal{G}$. 
Assume that $t(T)$, $t(T')$ and $t(db_{1,2}(db_0(T^{(N)})''))$ are $p^M$, $p^{M+1}$ and $p^{4M+4+4N}$, respectively. 
As in the previous two proofs, using Jensen's inequality one obtains 

$$ \left(\frac{1}{\#{\overline{e}}}\#db_0(T^{(N)})'' \right)^2 \leq \frac{1}{\#{\overline{e}}}\#db_{1,2}(db_0(T^{(N)})'') $$
and
$$ \left(\frac{1}{\#v}\#T^{(N+1)} \right)^2 \leq \frac{1}{\#v}\#db_0(T^{(N)})''. $$
Thus, since $\#db_{1,2}(db_0(T^{(N)})'') = (p^{4M+4+4N}+o(1))n^{8N}$, it follows that $\#db_0(T^{(N)})'' \leq (p^{2M+2+2N}+o(1))n^{4N+1},$ and $\#T^{(N+1)} \leq (p^{M+1+N}+o(1))n^{2N+1}$. This allows us to use Lemma \ref{lemma:jensen} as follows. Since
$$\sum_{v}\frac{\#\{T \mid v\}}{\#T} = 1, $$
$$\sum_{v}\frac{\#\{T \mid v\}}{\#T} \#\{e \mid v\} = \frac{\#T'}{\#T} = (p+o(1))n,$$
and
$$\sum_{v}\frac{\#\{T \mid v\}}{\#T} (\#\{e \mid v\})^{N+1} = \frac{\#T^{(N+1)}}{\#T} \leq (p^{N+1}+o(1))n^{N+1},$$
the conditions of Lemma \ref{lemma:jensen} are met. Applying Lemma \ref{lemma:jensen} gives us that the set $B_v^e$ of the ``bad" vertices, i.e., those that are not in $(p+o(1))n$ edges has weight $o(1)$. 
In other words, $\sum_{v \in B_v^e} \frac{\#\{T \mid v\}}{\#T} = o(1)$, and therefore only $o(n^N)$ copies of $T$ have their $v$ belonging to $B_v^e$. 
But, every $T$ is in $O(n^{N-1})$ copies of $db_0(T)$.
Thus, only $o(n^{2N-1})$ copies of $db_0(T)$ have their $v$ belonging to $B_v^{e}$. 
However, each copy of $db_0(T)$ with $v \not\in B^e_v$ forms $(p^{2N+2}+o(1))n^{2N+2}$ unique copies of $db_0(T^{(N)})''$, since $v$ is in $(p+o(1))n$ edges.
Therefore, 
$$(\#db_0(T) - o(n^{2N-1}))(p^{2N+2}+o(1))n^{2N+2} \leq \#db_0(T^{(N)})'' \leq (p^{2M+2+2N}+o(1))n^{4N+1}, $$
and so $\#db_0(T) \leq (p^{2M}+o(1))n^{2N-1}$. 
This now feeds into another usage of Lemma \ref{lemma:jensen}. Since
$$\sum_v \frac{1}{\#v} = 1,$$
$$\sum_v \frac{1}{\#v}\#\{T \mid v\} = \frac{\#T}{\#v} = (p^{M}+o(1))n^{N-1},$$
and
$$\sum_v \frac{1}{\#v}(\#\{T \mid v\})^2 = \frac{\#db_0(T)}{\#v} \leq (p^{2M}+o(1))n^{2N-2},$$
the conditions of Lemma \ref{lemma:jensen} are met.
Lemma \ref{lemma:jensen} then implies that the set $B_v^T$ of ``bad" vertices, i.e., those vertices that are not in $(p^M+o(1))n^{N-1}$ copies of $T$ has weight $o(1)$.
The weight used is uniform, so $|B_v^T| = o(n)$. 
Observe also, that every vertex in $B_v^e \setminus B_v^T$ forms $(p^M+o(1))n^{N-1}$ unique copies of $T$.
But, there are only $o(n^N)$ copies of $T$ with $v \in B_v^e$, so one has that $|B_v^e \setminus B_v^T| = o(n)$. This means that $|B_v^e \cup B_v^T| = o(n)$, and in particular $|B_v^e| = o(n)$. \\

This now gives a lower bound on $\#e$, since each of the $(1 - o(1))n$ vertices not in $B_v^e$ provides $(p+o(1))n$ unique copies of $e$. Thus, $\#e \geq (1 - o(1))n(p+o(1))n = (p+o(1))n^2$. \\

Finally, an upper bound on $\#C_4$ can also be obtained by noting that since only $o(n)$ vertices are in $B_v^e \cup B_v^T$, then only $o(n^4)$ copies of $C_4$ have a vertex in that set. 
Each of the remaining $\#C_4 - o(n^4)$ copies of $C_4$ does not have a vertex in $B_v^e \cup B_v^T$, and thus, forms $((p+o(1))n)^{4N}((p^M+o(1))n^{N-1})^4$ unique copies of $db_{1,2}(db_0(T^{(N)})'')$.
Hence,
$$(\#C_4 - o(n^4))((p+o(1))n)^{4N}((p^M+o(1))n^{N-1})^4 \leq \# db_{1,2}(db_0(T^{(N)})'') = (p^{4M + 4 + 4N}+o(1))n^{8N},$$
and that gives $\#C_4 \leq (p^4+o(1))n^4$, which finishes the proof, again as $\{e, C_4\}$ is forcing.
\end{proof}

\section{Non-forcing pairs}
\label{section:non-forcing}

In this section, the pairs from the triples from Theorem \ref{theorem:explicit} will be shown not to be forcing.

\begin{restatable}{theorem}{nonforcing}
\label{theorem:non-forcing}
For $i \in \{1, 2, 3\}$, if $\{F_1, F_2\} \subset \mathcal{H}_i(T)$, then $\{F_1, F_2\}$ is not forcing.
\end{restatable}

\noindent
In this section, we shall reduce the proof of Theorem \ref{theorem:non-forcing} to the verification of a few simple inequalities. The verification itself is provided in Appendix \ref{section:verification}. \\

For each required pair of graphs, we need to construct a graph sequence that is not quasirandom but that has the correct homomorphism counts for the two graphs in the pair. 
We shall model these by finite weighted graphs with loops in the standard way, as was done by Lovász and Szegedy \cite{LS-04}. 
Let $\mathfrak{G} = (V, w)$ be a finite weighted graph with loops. Here, $V$ is the vertex set, and $w : V \times V \to [0, 1]$ is symmetric weight function. 
For every graph $F$, define $$t_\mathfrak{G}(F) := \sum_{\phi : V(F) \to V(\mathfrak{G})} \frac{1}{|V(\mathfrak{G})|^{|V(F)|}} \prod_{e = \{u, v\} \in E(F)}w(\phi(u), \phi(v)).$$
Having $\mathfrak{G}$ as above, one can obtain a sequence $\mathcal{G} = \{G_n\}_{n\geq0}$, such that almost surely $t_\mathcal{G}(F)$ is equal to $ t_\mathfrak{G}(F)$ for all finite graphs $F$. This is done as follows. First, take a function $\phi : V(G_n) \to V(\mathfrak{G})$ uniformly at random. 
Then, for each distinct $u, v \in V(G_n)$ connect them by an edge randomly and independently with probability $w(\phi(u), \phi(v))$. 
Using Azuma's inequality, one then shows that for any graph $F$, the value $t_\mathcal{G}(F)$ is equal to $ t_\mathfrak{G}(F)$ almost surely. 
Since there are countably many finite graphs up to isomorphism, the above holds simultaneously for all graphs $F$ almost surely. 
Having this in mind, we shall just discuss weighted graphs $\mathfrak{G}$, and will assume that $t(F) = t_\mathfrak{G}(F)$ since one can always generate an appropriate sequence as described. 
Note that a sequence $\mathcal{G}$ generated in this way will be quasirandom if and only if the weighted graph $\mathfrak{G}$ has a weight function $w$ that is constant almost everywhere.
\\

Before proceeding further, we introduce one further piece of notation. 
Given a graph $F$, define $f(F) := t(F)^{\frac{1}{|E(F)|}}$, and observe that if $\mathcal{G}$ is a quasirandom sequence with density $p \in (0, 1)$, and the value $t(F)$ is known, then necessarily $p = f(F)$. 
The following proposition now follows immediately from the definition of forcing.

\begin{proposition}
\label{prop:non-forcing}
Let $F_1$ and $F_2$ be graphs, and let $\mathcal{G}$ be a sequence of graphs. If $\{F_1, F_2\}$ is forcing, and $f(F_1) = f(F_2) = p \in (0, 1)$, then $\mathcal{G}$ must be quasirandom with density $p$. Equivalently, if $\mathcal{G}$ is not quasirandom and $f(F_1) = f(F_2) = p \in (0, 1)$, then  $\{F_1, F_2\}$ is not forcing.
\end{proposition}

Note that, since the values of the weight function $w$ determine $\mathfrak{G}$, we can consider $\mathfrak{G}$, $t(F)$ and $f(F)$, for any $F$, to be functions of  $w$.
When it is important to emphasise which $w$ is used, we write $\mathfrak{G}_w$, $t_w(F)$ and $f_w(F)$. 

Since the values $t(F)$ and $f(F)$ are continuous
functions of $w$ the following is immediate.

\begin{proposition}
\label{prop:path}
If for some symmetric $w_1, w_2 \in [0, 1]^{n \times n}$, and some graphs $F_1$ and $F_2$ one has
$$f_{w_1}(F_1) < f_{w_1}(F_2)$$ 
and 
$$f_{w_2}(F_1) > f_{w_2}(F_2),$$
then any path from $w_1$ to $w_2$ contains  a symmetric  $w\in [0, 1]^{n \times n}$ such that 
$$ f_{w}(F_1) = f_{w}(F_2). $$
In particular, if $n > 1$, the set $\{F_1, F_2\}$ is not forcing.
\end{proposition}

The last point of Proposition \ref{prop:path} boils down to the fact that any two points $x, y \in [0, 1]^k$, where $k > 2$, can be connected by a path whose only possible intersections with the line $D = \{(c, \dots, c) : c \in [0, 1]\}$ are the endpoints $x$ and $y$ themselves. 
Thus, by Proposition \ref{prop:path}, there is a $w$ for which $f_w(F_1) = f_w(F_2)$, but that is not on $D$. As discussed, only $w \in D$ correspond to quasirandom sequences. Therefore, the pair $\{F_1, F_2\}$ is not forcing, by Proposition \ref{prop:non-forcing}. 

\subsection{Weight functions}

While using Proposition \ref{prop:path} to prove Theorem \ref{theorem:non-forcing}, only two types of the weights $w$ need to be considered.
Both mentioned weights correspond to $\mathfrak{G}_w$ with only two vertices, say $V(\mathfrak{G}_w) = \{\mathfrak{0}, \mathfrak{1}\}$. 
Therefore, $w$ is determined by the values of $w(\mathfrak{0}, \mathfrak{0})$, $w(\mathfrak{0}, \mathfrak{1})$ and $w(\mathfrak{1}, \mathfrak{1})$. 
With this in mind, $w$ will be written as an array $(w(\mathfrak{0}, \mathfrak{0}), w(\mathfrak{0}, \mathfrak{1}), w(\mathfrak{1}, \mathfrak{1}))$. 
Having said this, the first weight function is going to be $(1, 0, 1)$, and the second one is going to be of the form $(x, 1, x)$ for some small $x > 0$.\\

Fix an arbitrary connected $F$ and set $n := n(F)$ and $m := m(F)$. 
Note that in the case $w = (0, 0, 1)$, one has $t(F) = \left(\frac{1}{2} \right)^n$, since the map is a homomorphism if and only if it maps all the vertices of $F$ into the second vertex. 
Similarly, in the case $w = (1, 0, 1)$ one has $t(F) = 2\left(\frac{1}{2} \right)^{n}$, as $F$ is connected. Thus, $f(F) = \left(\frac{1}{2} \right)^{\frac{n-1}{m}}$. \\

On the other hand, in the case $w = (x, 1, x)$, if $F$ is non-bipartite, then $t(F) \to 0$ as $x \to 0$. 
Thus, we look at the asymptotics of $t(F)$, or rather of $f(F)$, as $x \to 0$. 
Once more fix some arbitrary graph $F$, and set $n := n(F)$, $m := m(F)$ and $b := b(F)$. 
Recall that $b(F) := \text{max-cut}(F)$ is the number of edges in the largest bipartite subgraph of $F$,
and observe that $t(F) = \Theta(x^{m - b})$, since 
by definition 
$$t(F) = \sum_{\phi : V(F) \to V(\mathfrak{G})} \frac{1}{|V(\mathfrak{G})|^{|V(F)|}} \prod_{e = \{u, v\} \in E(F)}w(\phi(u), \phi(v)).$$
This, by the choice of $w$, shows also that $t(F)$ is a polynomial in $x$. Furthermore, by the definition of $b$ there is a bipartition of $F$ such that exactly $b$ edges have a vertex in each part.
Therefore, there is a map $\phi$ that contributes $\frac{1}{2^{n}}x^{m - b}$ to the sum.
Since  all contributions to the sum are positive monomials, this means that $t(F) = \Omega(x^{m - b})$. On the other hand, $t(F) = O(x^{m-b})$, by the maximality of $b$. 
Finally, this also implies that $f(F) = \Theta\left(x^{\frac{m - b}{m}}\right) = \Theta\left(x^{1 - \frac{b}{m}}\right)$. 
In particular, if $F$ is bipartite, then $b = m$, and thus $f(F) = \Theta(1)$.\\

These types of weight functions can be used, together with Proposition \ref{prop:path}, to show that the pairs of graphs from the statement of Theorem \ref{theorem:non-forcing} are not forcing. For convenience, we restate everything in the following proposition. 

\begin{proposition}
\label{prop:condition}
Let $F_1$ and $F_2$ be connected graphs. For $i \in \{1, 2\}$, let $n_i = n(F_i)$, $m_i = m(F_i)$ and $b_i = b(F_i)$.
\begin{itemize}
    \item [-]If $\frac{n_1 - 1}{m_1} > \frac{n_2 - 1}{m_2}$, then there is $w$ such that $f_w(F_1) < f_w(F_2)$.
    \item [-]If $\frac{b_1}{m_1} > \frac{b_2}{m_2}$, then there exists $w$ such that $f_w(F_1) > f_w(F_2)$. 
    \item [-]In particular, if both hold, the pair $\{F_1, F_2\}$ is not forcing.
\end{itemize}

\end{proposition}

\noindent As discussed, the fact that this proposition implies Theorem \ref{theorem:non-forcing} is just a matter  of verifying simple inequalities, which is done in Appendix \ref{section:verification}.

\section{Concluding remarks}

Forcing triples with no forcing pairs could also be seen as minimal forcing sets, since none of their proper subsets are forcing sets. It is of course natural to consider minimal forcing sets which contain more elements. In a follow up paper we show the existence of minimal forcing sets containing any finite number of graphs. \\

Another question of interest is the following. If the forcing conjecture fails for some graph $B$, i.e. if a bipartite graph $B$ contains a cycle and $\{e, B\}$ is not forcing, is there a graph $H$ for which $\{e, B, H\}$ is a forcing triple with no forcing pairs?

\section{Acknowledgments}

I would like to thank Gavrilo Milićević and Victor Souza for many helpful discussions, and my supervisor Timothy Gowers for his continual support and guidance.

\bibliographystyle{plain}
\bibliography{auxiliary/references}

\appendix

\section{Approximate equality in Jensen's inequality}
\label{section:jensen}
We now give a proof of Lemma \ref{lemma:jensen}, which was used in the proofs of Lemma \ref{lemma:triple_second} and Lemma \ref{lemma:triple_third}.
The lemma is similar to a lemma concerning approximate equality in the Cauchy-Schwarz inequality that was presented by Kohayakawa, R\"{o}dl and Sissokho \cite{KRS_04} and Skokan and Thoma \cite{ST-04}. The only difference between this version and the one from \cite{ST-04} is the fact that here an arbitrary weighted average is considered. The proof is almost identical to the one given in \cite{ST-04}, and is provided here just for the sake of completeness.

\jensen*



\begin{proof}
To begin with, we prove the case $k = 2$. The cases $k > 2$ will follow from this case by Jensen's inequality. Let $k = 2$. Then 

\begin{equation*}
\begin{split}
    \sum_{i=0}^n p_i (a_i - A)^2 
    & = \sum_{i=0}^n p_i a_i^2 -2A\sum_{i=0}^n p_i a_i + A^2 \leq \\
    & \leq A^2((1 + \epsilon) -2 (1 - \epsilon) + 1) = \\
    & = A^2(3\epsilon).
\end{split}
\end{equation*}



\noindent
On the other hand, by the definition of $I$, if $\sum_{i\in I}t(i) \geq \delta$, then

$$\sum_{i=0}^n p_i(a_i -A)^2 \geq  \sum_{i\in I} p_i (a_i - A)^2 \geq \delta^2  A^2 \sum_{i\in I}p_i \geq \delta^3 A^2.$$
It follows that $\delta^3\leq 3\epsilon$. Therefore, the case $k = 2$ is done.  \\

To see that the theorem is true when $k > 2$, observe that, since then the function $x^{\frac{2}{k}}$ is concave, one has that

$$\sum_{i=0}^n p_i a_i^2 = \sum_{i=0}^n p_i (a_i^k)^{\frac{2}{k}} \leq \left(\sum_{i=0}^n p_i a_i^k \right)^{\frac{2}{k}}.$$
Therefore, if  $$\sum_{i=0}^n p_i a_i^k \leq A^k (1 + \epsilon),$$ then also
$$\sum_{i=0}^n p_i a_i^2 \leq A^2 (1 + \epsilon)^{\frac{2}{k}} \leq A^2(1 + \epsilon).$$
Thus, the problem has been reduced to the case $k = 2$, and the proof is finished.
\end{proof}




\section{Simple inequalities}
\label{section:verification}

In this section, we complete the proof of Theorem \ref{theorem:non-forcing} by verifying the inequalities needed for Proposition \ref{prop:condition}.

\nonforcing*


\begin{proof}
To see that each pair is not forcing, we draw up a table of various parameters applied to the graphs we use (see Table \ref{table:graphs} below). We shall use these in conjunction with Proposition \ref{prop:condition}.

Before we do this, we need to justify the values of $b$ in the table. Let $F$ be a graph with some vertex labelled $0$, and consider how $b(F)$ behaves when we construct graphs using the methods described in Section \ref{section:notation}. 
It is clear that $b(F^{(k)}) = b(F) + k$, $b(db_0(F)) = 2 b(F)$, $b(db_{0, 1}(F')) = 2b(F)+1$ and that $b(db_{1,2}(F''))=2b(F) + 4$. This gives us what we need.

Now let us prove the necessary inequalities.
For convenience, let $g_1(\cdot) := \frac{n(\cdot) - 1}{m(\cdot)}$, and $g_2(\cdot) := \frac{b(\cdot)}{m(\cdot)}$. 
Observe also, that the simple inequalities $N \leq M$, $B < M$ and $3 \leq N$ hold since $T$ is connected and non-bipartite. \\

\begin{table}
\begin{center}
 \begin{tabular}{| c | c | c | c | c | c |}
 \hline
 graph & $n$ & $m$ & $b$ & $\frac{n-1}{m}$ & $\frac{b}{m{}}$ \\ [0.5ex]
 \hline
 \hline
 $e$ & $2$ & $1$ & $1$ & $1$ & $1$ \\ [0.5ex]
 \hline
 $C_{2N}$ & $2N$ & $2N$ & $2N$ & $\frac{2N-1}{2N}$ & $1$ \\ [0.5ex]
 \hline
 $T$ & $N$ & $M$ & $B$ & $\frac{N-1}{M}$ & $\frac{B}{M}$ \\ [0.5ex]
 \hline
 $T'$ & $N+1$ & $M+1$ & $B+1$  & $\frac{N}{M+1}$ & $\frac{B+1}{M+1}$ \\ [0.5ex]
 \hline 
 $db_{0,1}(T')$ & $2N$ & $2M+1$ & $2B+1$ & $\frac{2N-1}{2M+1}$ & $\frac{2B+1}{2M+1}$ \\ [0.5ex] 
 \hline 
 $db_{1, 2}(db_0(T^{(N)})'')$ & $8N$ & $4M+4+4N$ & $4B+4+4N$ & $\frac{8N-1}{4M+4+4N}$ & $\frac{4B+4+4N}{4M+4+4N}$ \\ [0.5ex]
 \hline 
 $db_{1,2}(db_0^2(T)'')$ & $8N-4$ & $8M+4$ & $8B+4$  & $\frac{8N-5}{8M+4}$ & $\frac{8B+4}{8M+4}$ \\ [0.5ex]
 \hline
\end{tabular}
\caption{Overview of the parameters of used graphs.}
\label{table:graphs} 
\end{center}
\end{table}

Consider first the triple $\mathcal{H}_1(T)=\{C_{2N}, T', db_{0,1}(T')\}$ from Lemma \ref{lemma:triple_first}. We have that
$$ g_1(C_{2N}) = \frac{2N - 1}{2N} > \frac{N}{N+1} \geq \frac{N}{M+1} = g_1(T'), $$
$$ g_1(T') = \frac{N}{M+1} = \frac{2N}{2M+2} > \frac{2N-1}{2M+1} = g_1(db_{0,1}(T')),$$
$$ g_2(C_{2N}) = 1 > \frac{B+1}{M+1} = g_2(T'), $$
and
$$ g_2(T') = \frac{B+1}{M+1} = \frac{2B+2}{2M+2} > \frac{2B+1}{2M+1} = g_2(db_{0,1}(T')).$$ 
Therefore, the pairs are not forcing, by Proposition \ref{prop:condition}.\\

Next, we consider the triple $\mathcal{H}_2(T)=\{e, T', db_{1,2}(db_0^2(T)'')\}$ from Lemma \ref{lemma:triple_second}.  Here we have the inequalities
$$ g_1(e) = 1 >  \frac{N}{M+1} = g_1(T'),$$
$$ g_1(T') = \frac{N}{M+1} = \frac{8N}{8M+8} > \frac{8N-4}{8M+4} > \frac{8N-5}{8M+4} = g_1(db_{1,2}(db_0^2(T)'')),$$
$$ g_2(e) = 1 > \frac{B+1}{M+1} = g_2(T'), $$
and
$$ g_2(T') = \frac{B+1}{M+1} = \frac{8B+8}{8M+8} > \frac{8B+4}{8M+4} = g_2(db_{1,2}(db_0^2(T)'')).$$ 
As in the previous case, the pairs are not forcing, by Proposition \ref{prop:condition}.\\

Finally, for the triple $\mathcal{H}_3(T)=\{db_{1, 2}(db_0(T^{(N)})''), T', T\}$ from Lemma \ref{lemma:triple_third}, we have the inequality
$$ g_1(db_{1, 2}(db_0(T^{(N)})'')) = \frac{8N - 1}{4M + 4 + 4N} > \frac{N}{M+1} = g_1(T'),$$
because it is equivalent to
$$ (M+1)(4N-1) > 4N^2, $$
which is true, because $M \geq N > 1$. 
Furthermore,
$$ g_1(T') = \frac{N}{M+1} > \frac{N-1}{M} = g_1(T),$$
$$ g_2(db_{1, 2}(db_0(T^{(N)})'')) = \frac{4B+4+4N}{4M+4+4N} > \frac{4B+4}{4M+4} = \frac{B+1}{M+1} = g_2(T'),$$
and
$$ g_2(T') = \frac{B+1}{M+1} > \frac{B}{M} = g_2(T). $$ 
Therefore, the pairs are not forcing, by Proposition \ref{prop:condition} again, and Theorem \ref{theorem:non-forcing} is proved.
\end{proof}

\end{document}